\documentclass[12pt,a4paper]{article}
\usepackage{amsmath,amsfonts,amssymb,amsthm}
\usepackage{mathrsfs}
\usepackage{ifthen,textcomp,multirow,longtable}
\usepackage{hyperref}

\newtheorem{Theorem}{Theorem}

\newtheorem{Remark}{Remark}

\newcommand\emp\varnothing
\newcommand\eps\varepsilon
\newcommand\ov\widetilde

\def\^#1{^{\overline{#1}}}

\begin{document}

\title{Left-invariant pseudo-Riemannian metrics\\ on four-dimensional Lie groups\\ with zero Schouten-Weyl tensor}

\author{Olesya P. Khromova, Pavel N. Klepikov, Eugene D. Rodionov\thanks{This work was supported by the Russian Foundation for Basic Research (projects \textnumero 16--01--00336a, \textnumero 16--31--00048mol\_a), and by the Ministry of Education and Science of Russian Federation in the framework of the base part of a government order to Altai State University in the area of scientific activity (project no. 1148).}}

\date{}

\maketitle

\begin{abstract}
In the presented paper left-invariant pseudo-Riemannian metrics on four-dimensional Lie groups with zero Schouten-Weyl tensor are investigated. The complete classification of these metric Lie groups is obtained in terms of the structure constants of corresponding Lie algebras.
\end{abstract}

\allowdisplaybreaks

\section{Introduction}

Riemannian manifolds with zero Schouten-Weyl tensor were investigated by~many mathematicians (see for example~\cite{Besse}). In particular, this class contains Einstein manifolds ($ r = \lambda g $) and their direct products, locally symmetric spaces ($ \nabla R = 0 $), Ricci parallel manifolds ($ \nabla r = 0 $), and conformally flat manifolds ($ W = 0 $). 

In general case the classification problem of (pseudo)Riemannian mani\-folds with zero Schouten-Weyl tensor is very difficult. Therefore one can consider some restrictions. So, for example, G. Calvaruso and A. Zaeim have classified pseudo-Riemannian left-invariant Einstein metrics, conformally flat metrics and metrics with parallel Ricci tensor on four-dimensional Lie groups~\cite{2,3,4}. Besides this, O.P. Khromova, E.D. Rodionov and V.V.~Slavskii have classified four-dimensional Lie groups with left-invariant Riemannian metric and zero divergence Weyl tensor~\cite{5,6,7,9}. Also, A.~Zaeim and A.~Haji-Badali have classified Einstein-like pseudo-Riemannian homogeneous 4-manifolds with nontrivial isotropy~\cite{ZHb}.

In this paper, we obtain full classification of left-invariant pseudo\nobreakdash-Riemannian metrics on four-dimensional Lie groups with zero Schouten\nobreakdash-Weyl tensor, which are neither Einstein metrics, nor conformally flat metrics, nor Ricci parallel metrics.

The Schouten-Weyl tensor $SW$ on the (pseudo)Riemannian manifold $(M,g)$ of dimension $n\geqslant3$ is defined by the formula
\begin{equation*}
SW(X,Y,Z)=\nabla_Z A(X,Y)- \nabla_Y A(X,Z),
\end{equation*}
were $A = \frac{1}{n - 2}\left(r - \frac{sg}{2(n - 1)}\right)$ is the one-dimensional curvature tensor, $s$ is the~scalar curvature.
If $n\geqslant4$, then The Schouten-Weyl tensor is connected with the divergence of the Weyl tensor via the next equation~\cite{Besse}:
\begin{equation*}
SW = -(n-3)\,\mathrm{div}\,W.
\end{equation*}
If the scalar curvature of the (pseudo)Riemannian manifold is constant, then the following conditions are equivalent
\begin{equation}\label{eq:r_ijk=r_ikj} 
SW=0 \quad \Leftrightarrow \quad \nabla_Z r(X,Y) = \nabla_Y r(X,Z). 
\end{equation}

Following the standard terminology (see for example \cite{2,13}), we define {\it Segre type} of a given a self-adjoint operator with respect to a nondegenerate inner product, which are listed between brackets $\{$ $\}$, and they denote the~sizes of Jordan cells in the decomposition of the operator. Round brackets group together different blocks, which refer to the one eigenvalue. When different blocks refer to the one eigenvalue, the Segre type is said to be {\it degenerate}. 

A fundametal step for the problem of classification of four-dimensional metric Lie groups with zero Schouten-Weyl tensor is to define which Segre types of the Ricci operator are possible, and to find corresponding Lie algebras. All possible Segre types of the Ricci operator on four-dimensional metric Lie groups are listed in the Table~\ref{table:Segre_types}.

\begin{table}[h] 
\caption{Possible Segre types of the Ricci operator on four-dimensional metric Lie groups.}
\label{table:Segre_types}
	\centering
\begin{tabular}{|c|c|c|c|c|c|}
\hline 
Nondegenerate               & $\{1111\}$     & $\{112\}$   & $\{22\}$   & $\{13\}$   & $\{4\}$ \\
\hline
\multirow{4}{*}{Degenerate} & $\{11(11)\}$   & $\{1(12)\}$ & $\{(22)\}$ & $\{(13)\}$ & --- \\
                            & $\{(11)(11)\}$ & $\{(11)2\}$ &            &            &     \\
                            & $\{1(111)\}$   & $\{(112)\}$ &            &            &     \\
                            & $\{(1111)\}$   &             &            &            &     \\
\hline
\hline
Nondegenerate & $\{111\overline{1}\}$   & $\{21\overline{1}\}$ & $\{2\overline{2}\}$ & \multicolumn{2}{c|}{$\{1\overline{1}1\overline{1}\}$}   \\
\hline
Degenerate    & $\{(11)1\overline{1}\}$ & ---                  & ---                 & \multicolumn{2}{c|}{$\{(1\overline{1}1\overline{1})\}$} \\
\hline
\end{tabular}
\end{table}

\begin{Remark}
Note that Segre types $\{22\}$, $\{4\}$, $\{21\overline{1}\}$, $\{2\overline{2}\}$, $\{1\overline{1}1\overline{1}\}$ (and corresponding degenerate Segre types) are only possible in the case of metric of neutral signature, but not in the Lorentzian case.
\end{Remark}

\section{Classification of the four-dimensional metric Lie groups with zero Schouten-Weyl tensor}

In this part we consider pseudo-Riemannian left-invariant metrics on Lie groups with zero Schouten-Weyl tensor which are neither Einstein metrics, nor Ricci parallel metrics, nor conformally flat metrics, since G. Calvaruso and A. Zaeim have classified earlier pseudo-Riemannian left-invariant Einstein metrics, metrics with parallel Ricci tensor, and conformally flat metrics on four-dimensional Lie groups (see~\cite{2,3,4}).

Let us prove the following Theorem.

\begin{Theorem} \label{theorem:not}
Let $(G,g)$ be a four-dimensional metric Lie group with zero Schouten-Weyl tensor, and $(G,g)$ is neither Einstein, nor conformally flat, nor Ricci parallel. Then the Ricci operator $\rho$ has only the following Segre types:
$$ \{1(12)\}, \quad \{(11)2\}, \quad \{(112)\}, \quad \{(22)\}, \quad \{111\overline{1}\}. $$
\end{Theorem}

\begin{proof}
We prove this result from case-by-case, starting from the possible Segre types of the Ricci operator $\rho$, which are listed in the Table~\ref{table:Segre_types}.

Let $\rho$ has Segre type $\{(11)(11)\}$. Then there exists a basis $\{e_1,e_2,e_3,e_4\}$ in the metric Lie algebra such that the metric tensor $g$ and the Ricci tensor~$r$ have the following form (see~\cite{ONeil})
\begin{equation*}
r = \begin {pmatrix}
\varepsilon_1\rho_1 & 0                   & 0                   & 0 \\
0                   & \varepsilon_2\rho_1 & 0                   & 0 \\
0                   & 0                   & \varepsilon_3\rho_2 & 0 \\
0                   & 0                   & 0                   & \varepsilon_4\rho_2
\end {pmatrix},\quad 
g = \begin {pmatrix}
\varepsilon_1 & 0             & 0             & 0 \\
0             & \varepsilon_2 & 0             & 0 \\
0             & 0             & \varepsilon_3 & 0 \\
0             & 0             & 0             & \varepsilon_4
\end {pmatrix},
\end{equation*}
where $\rho_1\ne\rho_2$ and $\varepsilon_i=\pm1$.

Next, we consider the equations~(\ref{eq:r_ijk=r_ikj}) as restrictions on the structure constants of the Lie algebra:
\begin{align*}
\left( C_{12}^3\varepsilon_{{3}}-C_{13}^2\varepsilon_{{2}}-C_{23}^1\varepsilon_{{1}} \right)  \left( \rho_1-\rho_2 \right) &=0, & C_{12}^3\varepsilon_{{3}} \left( \rho_1-\rho_2 \right) &=0,\\
\left( C_{12}^3\varepsilon_{{3}}+C_{13}^2\varepsilon_{{2}}+C_{23}^1\varepsilon_{{1}} \right)  \left( \rho_1-\rho_2 \right) &=0, & C_{12}^4\varepsilon_{{4}} \left( \rho_1-\rho_2 \right) &=0,\\
\left( C_{12}^4\varepsilon_{{4}}-C_{14}^2\varepsilon_{{2}}-C_{24}^1\varepsilon_{{1}} \right)  \left( \rho_1-\rho_2 \right) &=0, & C_{13}^1\varepsilon_{{1}} \left( \rho_1-\rho_2 \right) &=0,\\
\left( C_{12}^4\varepsilon_{{4}}+C_{14}^2\varepsilon_{{2}}+C_{24}^1\varepsilon_{{1}} \right)  \left( \rho_1-\rho_2 \right) &=0, & C_{13}^3\varepsilon_{{3}} \left( \rho_1-\rho_2 \right) &=0,\\
\left( C_{13}^4\varepsilon_{{4}}+C_{14}^3\varepsilon_{{3}}-C_{34}^1\varepsilon_{{1}} \right)  \left( \rho_1-\rho_2 \right) &=0, & C_{14}^1\varepsilon_{{1}} \left( \rho_1-\rho_2 \right) &=0,\\
\left( C_{13}^4\varepsilon_{{4}}+C_{14}^3\varepsilon_{{3}}+C_{34}^1\varepsilon_{{1}} \right)  \left( \rho_1-\rho_2 \right) &=0, & C_{14}^4\varepsilon_{{4}} \left( \rho_1-\rho_2 \right) &=0,\\
\left( C_{23}^4\varepsilon_{{4}}+C_{24}^3\varepsilon_{{3}}-C_{34}^2\varepsilon_{{2}} \right)  \left( \rho_1-\rho_2 \right) &=0, & C_{23}^2\varepsilon_{{2}} \left( \rho_1-\rho_2 \right) &=0,\\
\left( C_{23}^4\varepsilon_{{4}}+C_{24}^3\varepsilon_{{3}}+C_{34}^2\varepsilon_{{2}} \right)  \left( \rho_1-\rho_2 \right) &=0, & C_{23}^3\varepsilon_{{3}} \left( \rho_1-\rho_2 \right) &=0,\\
C_{24}^2\varepsilon_{{2}} \left( \rho_1-\rho_2 \right) &=0, & C_{24}^4\varepsilon_{{4}} \left( \rho_1-\rho_2 \right) &=0,\\
C_{34}^1\varepsilon_{{1}} \left( \rho_1-\rho_2 \right) &=0, & C_{34}^2\varepsilon_{{2}} \left( \rho_1-\rho_2 \right) &=0.
\end{align*}

Solving this equations, we obtain that the metric Lie group must be Ricci parallel. We found analogously that Segre types different from $\{1(12)\}$, $\{(11)2\}$, $\{(112)\}$, $\{(22)\}$ and $\{111\overline{1}\}$ are not occur.
\end{proof}

\begin{Theorem} \label{theorem:1}
Let $(G,g)$ be a four-dimensional metric Lie group with zero Schouten-Weyl tensor, and $(G,g)$ is neither Einstein, nor conformally flat, nor Ricci parallel. Then the metric Lie algebra of the group $G$ is one from the Table~\ref{table:2}.
\end{Theorem}


{\footnotesize
\begin{longtable}{|c|p{210pt}|c|} 
\caption{Four-dimensional metric Lie algebras with zero Schou\-ten-Weyl tensor}
\label{table:2}
\endfirsthead
\endhead
\hline
Segre type  & Lie brakets & Metric tensor \\
\hline
$\{1(12)\}$ & 
\parbox[c]{210pt}{\centering $[e_2,e_3]=3ae_3$, $[e_2,e_4]=-{\frac {\varepsilon_2}{2a}}e_3+ae_4$, $a\ne0$.} & 
$\begin {pmatrix} \varepsilon_1 & 0 & 0 & 0 \\ 0 & \varepsilon_2 & 0 & 0 \\ 0 & 0 & 0 & \varepsilon_3 \\ 0 & 0 & \varepsilon_3 & 0 \end {pmatrix}$
\\
\cline{1-2}

$\{(11)2\}$
&
\parbox[c]{210pt}{\centering 
$[e_1,e_2]=\frac12cb\left(\delta\sqrt{5}-1 \right) e_1 - \frac12ca\left(\delta\sqrt {5}-1\right) e_2$,\\ $[e_1,e_3]=-\frac14ca\left(\delta\sqrt{5}-3\right) e_3$,\\ $[e_1,e_4]=ae_3+\frac14ca\left(\delta\sqrt{5}-3\right) e_4$,\\ $[e_2,e_3]=-\frac14cb\left(\delta\sqrt{5}-3 \right) e_3$,\\ $[e_2,e_4]=be_3+\frac14cb\left(\delta\sqrt{5}-3\right) e_4$,\\ $c=\frac{1}{a^2\varepsilon_1+b^2\varepsilon_2}$, $a^2\varepsilon_1+b^2\varepsilon_2\ne0$.} & \\

\cline{2-2}
 &
\parbox[c]{210pt}{\centering 
$[e_1,e_2]={\frac {2a \left( \varepsilon_1+\sqrt {5}\delta \right) }{\sqrt {5}\delta+3\varepsilon_{{1}}}}e_2$, $[e_1,e_3]=ae_3$, \\ 
$[e_1,e_4]={\frac { \left( \sqrt {5}\delta+3\varepsilon_{{1}} \right) }{4a}}e_{{3}}-ae_{{4}}$, $a\ne0$.} & \\

\cline{1-2}

$\{(112)\}$
&
\parbox[c]{210pt}{\centering 
$[e_1,e_2]=ae_3$, $[e_1,e_4]=be_1+ce_2+de_3$,\\ 
$[e_2,e_4]=fe_1+(\varphi-b)e_2+ge_3$, $[e_3,e_4]=\varphi e_3$, \\
$\varphi=-{\frac { \left( \varepsilon_{{1}}\varepsilon_{{2}}{a}^{2}-\varepsilon_{{1}}\varepsilon_{{2}}{c}^{2}-\varepsilon_{{1}}\varepsilon_{{2}}{f}^{2}-4{b}^{2}-2cf-2\varepsilon_{{3}} \right) }{4b}}\ne0$,\\ 
$a^2+c^2+f^2\ne0$, $b\ne0$.} & \\

\cline{2-2}
&
\parbox[c]{210pt}{\centering 
$[e_1,e_4]= \left( -\varepsilon_{{1}}\varepsilon_{{2}}c+\delta\varepsilon_{{1}}\varepsilon_{{2}}\sqrt {{a}^{2}-2\varepsilon_{{1}}\varepsilon_{{2}}\varepsilon_{{3}}}\right)e_{{2}}+be_{{3}}$,\\ 
$[e_{{1}},e_{{2}}]=ae_{{3}}$, $[e_{{2}},e_{{4}}]=ce_{{1}}+de_{{3}}+fe_{{2}}$,\\
$[e_{{3}},e_{{4}}]=fe_{{3}}$,\\
${a}^{2}-2\varepsilon_{{1}}\varepsilon_{{2}}\varepsilon_{{3}}\geqslant0$, $4c^2(a\varepsilon_1-c\varepsilon_3)^2 \ne 1$, $f\ne0$.} & \\

\cline{2-2}
&
\parbox[c]{210pt}{\centering 
$[e_{{1}},e_{{4}}]=ae_{{1}}+\psi e_{{2}}+be_{{3}}$,\\ $[e_{{2}},e_{{4}}]=ce_{{1}}+de_{{2}}+fe_{{3}}$, $[e_{{3}},e_{{4}}]=ge_{{3}}$, \\
$\psi=-\varepsilon_{{1}}\varepsilon_{{2}}c+\delta\varepsilon_{{1}}\varepsilon_{{2}}\sqrt {2}\sqrt {-\varepsilon_{{1}}\varepsilon_{{2}} \left({a}^{2}+{d}^{2}- \left( a+d \right) g+\varepsilon_{{3}} \right) }$, ${a}^{2}+{d}^{2}- \left( a+d \right) g+\varepsilon_{{3}}\geqslant0$, \\
$-{a}^{2}\varepsilon_{{1}}+a\varepsilon_{{1}}g+\varepsilon_{{2}}{c}^{2}+{d}^{2}\varepsilon_{{1}}-dg\varepsilon_{{1}}-\varepsilon_{{2}}{\psi}^{2}\ne0$,\\
$-2ac\varepsilon_{{1}}+gc\varepsilon_{{1}}-2d\psi\varepsilon_{{2}}+g\psi\varepsilon_{{2}}\ne0$, $g\ne0$.} & \\

\cline{2-2}
&
\parbox[c]{210pt}{\centering 
$[e_{{1}},e_{{2}}]=ae_{{3}}$, $[e_{{1}},e_{{4}}]=be_{{1}}+\psi e_{{2}}+ce_{{3}}$,\\ 
$[e_{{2}},e_{{4}}]=de_{{1}}+fe_{{2}}+ge_{{3}}$, $[e_{{3}},e_{{4}}]= \left( b+f \right) e_{{3}}$, \\
$\psi=-\varepsilon_{{2}}\varepsilon_{{1}}d+\delta\varepsilon_{{2}}\varepsilon_{{1}}\sqrt {4bf\varepsilon_{{1}}\varepsilon_{{2}}-2\varepsilon_{{1}}\varepsilon_{{2}}\varepsilon_{{3}}+{a}^{2}}$,\\
$4bf\varepsilon_{{1}}\varepsilon_{{2}}-2\varepsilon_{{1}}\varepsilon_{{2}}\varepsilon_{{3}}+{a}^{2}\geqslant0$,\\
$-\varepsilon_{{2}}ad\varepsilon_{{1}}\varepsilon_{{3}}-a\psi\varepsilon_{{3}}+{d}^{2}\varepsilon_{{2}}-\varepsilon_{{2}}{\psi}^{2}\ne0$,\\ 
$ba\varepsilon_{{3}}-fa\varepsilon_{{3}}-bd\varepsilon_{{1}}+b\psi\varepsilon_{{2}}+fd\varepsilon_{{1}}-f\psi\varepsilon_{{2}}\ne0$, \\
$b+f\ne0$.} & \\

\cline{2-3}
&
\parbox[c]{210pt}{\centering 
$[e_{{1}},e_{{2}}]=ae_{{1}}+\delta ae_{{2}}+be_{{3}}$,\\
$[e_{{1}},e_{{4}}]=ce_{{1}}+c\delta e_{{2}}+\psi e_{{3}}$,\\
$[e_{{2}},e_{{4}}]=de_{{1}}+d\delta e_{{2}}+fe_{{3}}$, $[e_{{3}},e_{{4}}]=\varphi e_{{3}}$, \\
$\psi=-{\frac {{b}^{3}+ \left( 2af\varepsilon_{{1}}\varepsilon_{{3}}-2c\delta d-{c}^{2}-{d}^{2}+2\varepsilon_{{3}} \right) b+2af \left(c\delta+d \right) }{2a \left( b\delta\varepsilon_{{1}}\varepsilon_{{3}}+d\delta+c \right) }}$,\\
$\varphi={\frac { 2bc\delta\varepsilon_{{1}}\varepsilon_{{3}}+2bd\varepsilon_{{1}}\varepsilon_{{3}}+2c\delta d+{b}^{2}+{c}^{2}+{d}^{2}+2\varepsilon_{{3}}  }{2\left(b\delta\varepsilon_{{1}}\varepsilon_{{3}}+d\delta+c\right)}}\ne0$,\\
$((-4abf+{b}^{2}d+{c}^{2}d-{d}^{3}-2d\varepsilon_{{3}} ) \varepsilon_{{1}}
+(-{b}^{3}+ ( {c}^{2}+{d}^{2} ) b-4adf ) \varepsilon_{{3}}-2b ) \delta-c (4af\varepsilon_{{3}}+{b}^{2}\varepsilon_{{1}}-2bd\varepsilon_{{3}}-{c}^{2}\varepsilon_{{1}}+{d}^{2}\varepsilon_{{1}}-2\varepsilon_{{1}}\varepsilon_{{3}} ) \ne0$.} & 
$\begin {pmatrix} \varepsilon_1 & 0 & 0 & 0 \\ 0 & -\varepsilon_1 & 0 & 0 \\ 0 & 0 & 0 & \varepsilon_3 \\ 0 & 0 & \varepsilon_3 & 0 \end {pmatrix}$
 \\

\cline{2-2}
&
\parbox[c]{210pt}{\centering 
$[e_{{1}},e_{{2}}]=-\delta ae_{{1}}+ae_{{2}}$, $[e_{{1}},e_{{3}}]=ae_{{3}}$,\\
$[e_{{2}},e_{{3}}]=\delta ae_{{3}}$, $[e_{{1}},e_{{4}}]=-b\delta e_{{1}}+be_{{2}}+ce_{{3}}$,\\ 
$[e_{{2}},e_{{4}}]=-d\delta e_{{1}}+de_{{2}}+\psi e_{{3}}$, \\
$[e_{{3}},e_{{4}}]=-a\delta e_{{2}}\varepsilon_{{1}}\varepsilon_{{3}}+ae_{{1}}\varepsilon_{{1}}\varepsilon_{{3}}+fe_{{3}}$,\\ 
$\psi=\frac{1}{4a} (\delta\varepsilon_{{3}}\varepsilon_{{1}}{b}^{2}+\delta\varepsilon_{{3}}\varepsilon_{{1}}{d}^{2}-2\delta\varepsilon_{{3}}\varepsilon_{{1}}df-2bd\varepsilon_{{1}}\varepsilon_{{3}}+
+2\varepsilon_{{3}}\varepsilon_{{1}}bf+4\delta ca+2\delta\varepsilon_{{1}})$,\\
$a\ne0$, $f\ne0$, $d\delta+b\ne0$.} & \\

\hline
$\{111\overline{1}\}$ &
\parbox[c]{210pt}{\centering 
$[e_{{2}},e_{{3}}]=-a\delta\sqrt {3}e_{{3}}+ae_{{4}}$,\\ $[e_{{2}},e_{{4}}]=ae_{{3}}+a\delta\sqrt {3}e_{{4}}$,\\ $[e_{{3}},e_{{4}}]=2a\varepsilon_{{2}}\varepsilon_{{3}}e_{{2}}$, $a\ne0$.} &
$\begin {pmatrix} \varepsilon_1 & 0 & 0 & 0 \\ 0 & \varepsilon_2 & 0 & 0 \\ 0 & 0 & \varepsilon_3 & 0 \\ 0 & 0 & 0 & \varepsilon_3 \end {pmatrix}$ \\
\hline

$\{(22)\}$
&
\parbox[c]{210pt}{\centering 
$[e_2,e_3]= ae_{{1}}$, $[e_1,e_4]= \left( 2 a+3 b \right) e_{{1}}$, \\
$[e_2,e_4]=  \left( 2 a+2 b \right) e_{{2}}$, $[e_3,e_4]= ce_{{1}}+de_{{2}}+be_{{3}}$, \\
$c \ne 0$, $a\ne 0 $, $a+b \ne 0$, $2 {a}^{2}+4 ab=\varepsilon$.} & 
$\begin {pmatrix} 0 & 0 & \varepsilon & 0 \\ 0 & 0 & 0 & \varepsilon \\ \varepsilon & 0 & 0 & 0 \\ 0 & \varepsilon & 0 & 0 \end {pmatrix}$
\\

\cline{2-2}
& 
\parbox[c]{210pt}{\centering 
$[e_1,e_3]= ae_{{1}}+be_{{2}}$, $[e_2,e_3]= ce_{{1}}+de_{{2}}$, \\
$[e_1,e_4]= \left( 2 c-b \right) e_{{1}}+fe_{{2}}$, \\
$[e_2,e_4]=   \left( 2 f-d \right) e_{{1}}+he_{{2}}$, $[e_3,e_4]= xe_{{1}}+ye_{{2}}$, \\
${a}^{2}+{f}^{2}+{c}^{2}+{h}^{2} \ne 0$, \\
${a}^{2}{y}^{2}-2 abxy-2 af{x}^{2}+2 af{y}^{2}-2 ahxy-{b}^{2}{x}^{2}+4 bc{x}^{2}-2 bc{y}^{2}+2 {f}^{2}{y}^{2}-2 fcxy+2 fd{x}^{2}-4 fhxy+2 {c}^{2}{y}^{2}-{d}^{2}{y}^{2}+2 dhxy+{h}^{2}{x}^{2}-{y}^{2}\varepsilon\ne 0 $,\\
$-2 bf+bd+fc=0$, \\
$-2 af+2 ad-2 cb-2 {f}^{2}+4 fd+2 {c}^{2}-2 {d}^{2}=\varepsilon$, \\
$2 abf-2 afc-2 {b}^{2}c-2 b{f}^{2}+2 b{c}^{2}+2 fcd-b\varepsilon=0$,\\
$2 af-2 {f}^{2}-2 {c}^{2}+2 hc=\varepsilon$, \\
$af+{b}^{2}-2 cb+bh-fd=0$.\\ }
& \\

\cline{2-2}
& 
\parbox[c]{210pt}{\centering                    
$[e_1,e_3]= -{\frac { \left( 3 {c}^{2}+2 {d}^{2} \right) }{d}}e_{{1}}-ce_{{2}}$,\\
$[e_2,e_3]= {\frac {c \left( {c}^{2}-{d}^{2} \right) }{{d}^{2}}}e_{{1}}-{\frac { {c}^{2}+3 {d}^{2}  }{d}}e_{{2}}$, \\
$[e_1,e_4]= 2 {\frac {c \left( {c}^{2}+{d}^{2} \right) }{{d}^{2}}}e_{{1}}$, \\
$[e_2,e_4]= {\frac {{c}^{2}}{d}}e_{{1}}+{\frac {c \left( 2 {c}^{2}+3 {d}^{2} \right)}{{d}^{2}}}e_{{2}}$, \\
$[e_3,e_4]= ae_{{1}}+be_{{2}}+ce_{{3}}+de_{{4}}$, \\
$d \ne 0$, $c \ne 0 $, $ac+bd \ne 0$, \\
$2 {c}^{6}+6 {c}^{4}{d}^{2}+4 {c}^{2}{d}^{4}-{d}^{4}=0$, $\varepsilon = 1$.}
& \\

\cline{2-2}
& 
\parbox[c]{210pt}{\centering 
$[e_1,e_3]= \left( 2 a-2 b \right) e_{{1}}$, $[e_2,e_3]= \left( 2 a-3 b \right) e_{{2}}$, \\
$[e_1,e_4]= ae_{{2}}$, $[e_3,e_4]= ce_{{1}}+de_{{2}}+be_{{4}}$, \\
$a \ne 0$, $d \ne 0 $, $b \ne 0 $, $a\ne b$, $2 {a}^{2}-4 ab=\varepsilon$. }
& \\

\cline{2-2}
& 
\parbox[c]{210pt}{\centering 
$[e_1,e_3]= -{\frac {d \left( 3 {c}^{2}+2 {d}^{2} \right) }{{c}^{2}}} e_{{1}}-{\frac{{d}^{2}}{c}}e_{{2}}$, \\
$[e_2,e_3]= -2 {\frac {d \left( {c}^{2}+{d}^{2} \right) }{{c}^{2}}} e_{{2}}$, \\
$[e_1,e_4]= {\frac { \left( 3 {c}^{2}+{d}^{2} \right) }{c}}e_{{1}}+{\frac {d \left({c}^{2}-{d}^{2} \right) }{{c}^{2}}}e_{{2}}$, \\
$[e_2,e_4]= de_{{1}}+{\frac { \left( 2 {c}^{2}+3 {d}^{2} \right) }{c}}e_{{2}}$, \\
$[e_3,e_4]= ae_{{1}}+be_{{2}}+ce_{{3}}+de_{{4}}$, \\
$c \ne 0$, $d \ne 0 $, $ac+bd \ne 0$, ${c}^{2}-{d}^{2} \ne 0$, \\
$4 {c}^{4}{d}^{2}+6 {c}^{2}{d}^{4}+2 {d}^{6}-{c}^{4}=0$, $\varepsilon = 1$. }
& \\

\cline{2-2}
& 
\parbox[c]{210pt}{\centering $
[e_1,e_3]= {\frac { \left( 12 {a}^{6}+59 {a}^{4}+97 {a}^{2}+42 \right) \delta_2  }{ \left( 4 {a}^{4}+17 {a}^{2}+21 \right)\sqrt {4 {a}^{4}+10 {a}^{2}+6}}}e_{{1}}+{\frac {\delta_1  \delta_2  a}{\sqrt {4 {a}^{4}+10 {a}^{2}+6}}}e_{{2}}$, \\
$[e_2,e_3]= -2 {\frac { \left( 4 {a}^{6}+21 {a}^{4}+38 {a}^{2}+21 \right) \delta_1  \delta_2  a}{ \left( 4 {a}^{4}+17 {a}^{2}+21\right) \sqrt {4 {a}^{4}+10 {a}^{2}+6}}}e_{{1}}$, \\
$[e_1,e_4]= -{\frac { \left( 4 {a}^{6}+13 {a}^{4}+4 {a}^{2}-21 \right) \delta_1  \delta_2  a}{ \left( 4 {a}^{4}+17 {a}^{2}+21\right) \sqrt {4 {a}^{4}+10 {a}^{2}+6}}}e_{{1}}+{\frac { \left( {a}^{2}+3 \right) \delta_2  }{\sqrt {4 {a}^{4}+10 {a}^{2}+6}}}e_{{2}}$, \\
$[e_2,e_4]= -{\frac {\delta_2  {a}^{2}}{\sqrt {4 {a}^{4}+10 {a}^{2}+6}}}e_{{1}}-{\frac {\delta_1  \delta_2  a \left( 8 {a}^{6}+46 {a}^{4}+93 {a}^{2}+63 \right)}{ \left( 4 {a}^{4}+17 {a}^{2}+21 \right) \sqrt {4 {a}^{4}+10 {a}^{2}+6}}} e_{{2}}$, \\
$[e_3,e_4]= be_{{1}}+ce_{{2}}+{\frac {\delta_1  \delta_2   \left( {a}^{2}+2 \right) a}{\sqrt {4 {a}^{4}+10 {a}^{2}+6}}}e_{{3}}+{\frac {\delta_2   \left( {a}^{2}+2 \right) }{\sqrt {4 {a}^{4}+10 {a}^{2}+6}}}e_{{4}}$, \\
$ab+\delta_1  c \ne 0$, $a \ne 0 $, $\varepsilon = 1$. }
& \\

\cline{2-2}
& 
\parbox[c]{210pt}{\centering 
$[e_1,e_3]= {\frac {13 \delta_2 \sqrt {55}}{66}}e_{{1}}+{\frac {\delta_1 \sqrt {21}\delta_2 \sqrt {55}}{330}}e_{{2}}$, \\
$[e_2,e_3]= -\frac{2}{15} \delta_1 \sqrt {21}\delta_2 \sqrt {55}e_{{1}}$, \\
$[e_1,e_4]= -{\frac {2 \delta_1 \sqrt {21}\delta_2 \sqrt {55}}{33}}e_{{1}}+{\frac {4 \delta_2 \sqrt {55}}{55}}e_{{2}}$, \\
$[e_2,e_4]= -{\frac {7 \delta_2 \sqrt {55}}{110}}e_{{1}}-{\frac {3 \delta_1 \sqrt {21}\delta_2 \sqrt {55}}{22}}e_{{2}}$, \\
$[e_3,e_4]= ae_{{1}}+be_{{2}}+{\frac {23 \delta_1 \sqrt {21}\delta_2 \sqrt {55}}{330}}e_{{3}}+{\frac {23 \delta_2 \sqrt {55}}{330}}e_{{4}}$, \\
$\sqrt {21}a+\delta_1 b \ne 0$, $\varepsilon = -1$. }
& \\

\cline{2-2}
& 
\parbox[c]{210pt}{\centering 
$[e_1,e_3]= -{\frac {a \left( 3 {b}^{2}+2 {a}^{2} \right) }{{b}^{2}+2 {a}^{2}}}e_{{1}}-{\frac{b{a}^{2}}{{b}^{2}+2 {a}^{2}}}e_{{2}}$, \\
$[e_2,e_3]= -2 {\frac {b{a}^{2}}{{b}^{2}+2 {a}^{2}}}e_{{1}}-2 ae_{{2}}$, \\
$[e_1,e_4]= 3 {\frac {b \left( {b}^{2}+{a}^{2} \right) }{{b}^{2}+2 {a}^{2}}}e_{{1}}+{\frac {a \left( {b}^{2}+{a}^{2} \right) }{{b}^{2}+2 {a}^{2}}}e_{{2}}$, \\
$[e_2,e_4]= {\frac {a{b}^{2}}{{b}^{2}+2 {a}^{2}}}e_{{1}}+{\frac {b \left( 2 {b}^{2}+3 {a}^{2} \right) }{{b}^{2}+2 {a}^{2}}}e_{{2}}$, \\
$[e_3,e_4]= ce_{{1}}+de_{{2}}+be_{{3}}+ae_{{4}}$, \\
${a} \ne 0$, $cb+da \ne 0$, $\varepsilon = -1$,\\ 
$4 {b}^{4}{a}^{2}+10 {b}^{2}{a}^{4}+6 {a}^{6}-{b}^{4}-4 {b}^{2}{a}^{2}-4 {a}^{4}=0$. }
& \\

\cline{2-2}
& 
\parbox[c]{210pt}{\centering 
$[e_1,e_3]= -{\frac { \left( 3 {a}^{2}+2 {b}^{2} \right) b}{{a}^{2}-2 {b}^{2}}}e_{{1}}-{\frac{a{b}^{2}}{{a}^{2}-2 {b}^{2}}}e_{{2}}$, \\
$[e_2,e_3]= 2 {\frac {a{b}^{2}}{{a}^{2}-2 {b}^{2}}}e_{{1}}-2 {\frac {{a}^{2}b}{{a}^{2}-2 {b}^{2}}}e_{{2}}$, \\
$[e_1,e_4]= {\frac {a \left( 3 {a}^{2}-{b}^{2} \right) }{{a}^{2}-2 {b}^{2}}}e_{{1}}+{\frac {\left( {a}^{2}-3 {b}^{2} \right) b}{{a}^{2}-2 {b}^{2}}}e_{{2}}$, \\
$[e_2,e_4]= {\frac {{a}^{2}b}{{a}^{2}-2 {b}^{2}}}e_{{1}}+{\frac { \left( 2 {a}^{2}+3 {b}^{2} \right) a}{{a}^{2}-2 {b}^{2}}}e_{{2}}$, \\
$[e_3,e_4]= ce_{{1}}+de_{{2}}+ae_{{3}}+be_{{4}}$, \\
$b \ne 0$, $ca+db \ne 0$, ${a}^{2}-2 {b}^{2} \ne0$, \\
$12 {a}^{4}{b}^{2}+6 {a}^{2}{b}^{4}-6 {b}^{6}-{a}^{4}\varepsilon+4 {a}^{2}{b}^{2}\epsilon_{{1}}-4 {b}^{4}\varepsilon=0$. }
& \\

\cline{2-2}
& 
\parbox[c]{210pt}{\centering 
$[e_1,e_3]= {\frac {8 \delta_{{2}}\sqrt {42}}{21}}e_{{1}}+\frac{1}{21} \delta_{{1}}\sqrt {5}\delta_{{2}}\sqrt {21}e_{{2}}$, \\
$[e_2,e_3]= -\frac{2}{21} \delta_{{1}}\sqrt {5}\delta_{{2}}\sqrt {21}e_{{1}}+\frac{2}{21} \delta_{{2}}\sqrt {42}e_{{2}}$, \\
$[e_1,e_4]= -{\frac {\delta_{{1}}\sqrt {5}\delta_{{2}}\sqrt {21}}{105}}e_{{1}}+{\frac {13 \delta_{{2}}\sqrt {42}}{42}}e_{{2}}$, \\
$[e_2,e_4]= -\frac{1}{21} \delta_{{2}}\sqrt {42}e_{{1}}-{\frac {19 \delta_{{1}}\sqrt {5}\delta_{{2}}\sqrt {21}}{105}}e_{{2}}$, \\
$[e_3,e_4]= ae_{{1}}+be_{{2}}+{\frac {8 \delta_{{1}}\sqrt {5}\delta_{{2}}\sqrt {21}}{105}}e_{{3}}+{\frac {4 \delta_{{2}}\sqrt{42}}{21}}e_{{4}}$, \\
$2a+\sqrt {10}b\delta_{{1}} \ne 0$, $\varepsilon=-1$. }
& \\

\cline{2-2}
& 
\parbox[c]{210pt}{\centering 
$[e_1,e_3]= {\frac {\delta_{{2}}\sqrt {21}\sqrt {2} \left( 21+4 \sqrt {21} \right) }{126}}e_{{1}}+{\frac {\delta_{{1}}\sqrt {-10+10 \sqrt {21}}\delta_{{2}}\sqrt {21}\sqrt {2} \left( 21+\sqrt {21} \right) }{2520}}e_{{2}}$, \\
$[e_2,e_3]= -{\frac {\delta_{{1}}\sqrt {-10+10 \sqrt {21}}\delta_{{2}}\sqrt {21}\sqrt {2} \left( 21+\sqrt {21} \right) }{1260}}e_{{1}}+1/3 \delta_{{2}}\sqrt {2}e_{{2}}$, \\
$[e_1,e_4]= \left( {\frac {\delta_{{1}}\sqrt {-10+10 \sqrt {21}}\delta_{{2}}\sqrt {42}}{120}}-\frac{\delta_{{1}}\sqrt {-20+20 \sqrt {21}}\delta_{{2}}}{24}  \right) e_{{1}}+{\frac {\delta_{{2}}\sqrt {21}\sqrt {2} \left( 63+\sqrt {21} \right) }{252}}e_{{2}}$, \\
$[e_2,e_4]= -1/6 \delta_{{2}}\sqrt {2}e_{{1}}-{\frac {\delta_{{1}}\sqrt {-10+10 \sqrt {21}}\delta_{{2}}\sqrt {21}\sqrt {2} \left( 9+\sqrt {21}\right) }{360}}e_{{2}}$, \\
$[e_3,e_4]= ae_{{1}}+be_{{2}}+{\frac {\delta_{{1}}\sqrt {-10+10 \sqrt {21}}\delta_{{2}}\sqrt {21}\sqrt {2}}{60}}e_{{3}}+1/6 \delta_{{2}}\sqrt {42}e_{{4}}$, \\
$b \left( 1+\sqrt {21} \right) \sqrt {-10+10 \sqrt {21}}+20 a\delta_{{1}} \ne 0$,\\
$\varepsilon=-1$. }
& \\

\cline{2-2}
& 
\parbox[c]{210pt}{\centering 
$[e_1,e_3]= -{\frac {a \left( 3 b{c}^{2}+2 b{a}^{2}+3 {c}^{3}+2 c{a}^{2} \right) }{{c}^{3}}}e_{{1}}-{\frac {{a}^{2} \left( c+b \right) }{{c}^{2}}}e_{{2}}$, \\
$[e_2,e_3]= be_{{1}}-{\frac { \left( b{c}^{2}+2 b{a}^{2}+2 {c}^{3}+2 c{a}^{2} \right) a}{{c}^{3}}}e_{{2}}$, \\
$[e_1,e_4]= {\frac {  2 b{c}^{2}+b{a}^{2}+3 {c}^{3}+c{a}^{2} }{{c}^{2}}}e_{{1}}-{\frac {a \left( b{a}^{2}-{c}^{3}+c{a}^{2} \right) }{{c}^{3}}}e_{{2}}$, \\
$[e_2,e_4]= {\frac {a \left( c+b \right) }{c}}e_{{1}}+{\frac {  2 b{c}^{2}+3 b{a}^{2}+2 {c}^{3}+3 c{a}^{2} }{{c}^{2}}}e_{{2}}$, \\
$[e_3,e_4]= de_{{1}}+fe_{{2}}+ce_{{3}}+ae_{{4}}$, \\
$c \ne 0$, $c+b \ne0$, $dc+fa \ne 0$, $b{c}^{2}+b{a}^{2}+c{a}^{2} \ne0$\\
$2 {b}^{2}{c}^{6}+6 {b}^{2}{c}^{4}{a}^{2}+6 {b}^{2}{c}^{2}{a}^{4}+2 {b}^{2}{a}^{6}+4 b{c}^{7}+12 b{c}^{5}{a}^{2}+12 b{c}^{3}{a}^{4}+4 bc{a}^{6}+4 {c}^{6}{a}^{2}+6 {c}^{4}{a}^{4}+2 {c}^{2}{a}^{6}-{c}^{6}\varepsilon=0$. \\ }
& \\

\cline{2-2}
& 
\parbox[c]{210pt}{\centering 
$[e_1,e_3]= {\frac {a \left( 12 {b}^{4}+11 {b}^{2}{a}^{2}+2 {a}^{4} \right) }{{b}^{4}+8 {b}^{2}{a}^{2}+{a}^{4}}}e_{{1}}+{\frac { \left( 4 {b}^{2}+{a}^{2} \right) b{a}^{2}}{{b}^{4}+8 {b}^{2}{a}^{2}+{a}^{4}}}e_{{2}}$, \\
$[e_2,e_3]= -{\frac {b \left( 5 {b}^{4}+9 {b}^{2}{a}^{2}+{a}^{4} \right) }{{b}^{4}+8 {b}^{2}{a}^{2}+{a}^{4}}}e_{{1}}+{\frac {a \left( 3 {b}^{4}+{b}^{2}{a}^{2}+{a}^{4} \right) }{{b}^{4}+8 {b}^{2}{a}^{2}+{a}^{4}}}e_{{2}}$, \\
$[e_1,e_4]= -{\frac {{b}^{3} \left( 7 {b}^{2}-2 {a}^{2} \right) }{{b}^{4}+8 {b}^{2}{a}^{2}+{a}^{4}}}e_{{1}}+{\frac {a \left( {b}^{4}+12 {b}^{2}{a}^{2}+2 {a}^{4} \right) }{{b}^{4}+8 {b}^{2}{a}^{2}+{a}^{4}}}e_{{2}}$, \\
$[e_2,e_4]= -{\frac {a{b}^{2} \left( 4 {b}^{2}+{a}^{2} \right) }{{b}^{4}+8 {b}^{2}{a}^{2}+{a}^{4}}}e_{{1}}-{\frac {b \left( 8 {b}^{4}+14 {b}^{2}{a}^{2}+3 {a}^{4}\right) }{{b}^{4}+8 {b}^{2}{a}^{2}+{a}^{4}}}e_{{2}}$, \\
$[e_3,e_4]= ce_{{1}}+de_{{2}}+be_{{3}}+ae_{{4}}$, \\
${b}^{2}+{a}^{2} \ne 0$, $cb+da \ne 0$, \\
$30 {b}^{10}+78 {b}^{8}{a}^{2}-18 {b}^{6}{a}^{4}-78 {b}^{4}{a}^{6}-12 {b}^{2}{a}^{8}-{b}^{8}\varepsilon -16 {b}^{6}{a}^{2}\varepsilon -66 {b}^{4}{a}^{4}\varepsilon -16 {b}^{2}{a}^{6}\varepsilon -{a}^{8}\varepsilon =0$. }
& \\

\cline{2-2}
& 
\parbox[c]{210pt}{\centering 
$[e_1,e_3]= 3/2 {\frac { \left( 3 {a}^{2}+2 {b}^{2} \right) b}{{a}^{2}+2 {b}^{2}}}e_{{1}}+3/2 {\frac {a{b}^{2}}{{a}^{2}+2 {b}^{2}}}e_{{2}}$, \\
$[e_2,e_3]= -1/2 {\frac {a \left( 5 {a}^{2}+4 {b}^{2} \right) }{{a}^{2}+2 {b}^{2}}}e_{{1}}+1/2 be_{{2}}$, \\
$[e_1,e_4]= -1/2 {\frac {a \left( 4 {a}^{2}-{b}^{2} \right) }{{a}^{2}+2 {b}^{2}}}e_{{1}}+1/2 {\frac { \left( 2 {a}^{2}+7 {b}^{2} \right) b}{{a}^{2}+2 {b}^{2}}}e_{{2}}$, \\
$[e_2,e_4]= -3/2 {\frac {{a}^{2}b}{{a}^{2}+2 {b}^{2}}}e_{{1}}-3/2 {\frac { \left( 2 {a}^{2}+3 {b}^{2} \right) a}{{a}^{2}+2 {b}^{2}}}e_{{2}}$, \\
$[e_3,e_4]= ce_{{1}}+de_{{2}}+ae_{{3}}+be_{{4}}$, \\
${a}^{2}+{b}^{2} \ne 0$, $ca+db \ne 0$, \\
$5 {a}^{6}+7 {a}^{4}{b}^{2}-5 {a}^{2}{b}^{4}-7 {b}^{6}-2 {a}^{4}\varepsilon-8 {a}^{2}{b}^{2}\varepsilon-8 {b}^{4}\varepsilon=0$. }
& \\

\cline{2-3}
& 
\parbox[c]{210pt}{\centering 
$[e_2,e_3]= ae_{{1}}$, $[e_1,e_4]=  {\frac {  2 {a}^{2}-3 \varepsilon  }{4a}}e_{{1}}$, \\
$[e_2,e_4]=  {\frac {  2 {a}^{2}-\varepsilon  }{2a}}e_{{2}}$, $[e_3,e_4]= be_{{1}}+ce_{{2}}- {\frac {  2 {a}^{2}+\varepsilon  }{4a}}e_{{3}}$, \\
$a \ne 0$, $b \ne 0$, $2 {a}^{2}\ne\varepsilon$. }
& 
$\begin {pmatrix} 0 & 0 & \varepsilon & 0 \\ 0 & 0 & 0 & -\varepsilon \\ \varepsilon & 0 & 0 & 0 \\ 0 & -\varepsilon & 0 & 0 \end {pmatrix}$
\\

\cline{2-2}
& 
\parbox[c]{210pt}{\centering 
$[e_1,e_3]=  {\frac {  2 {a}^{2}+\varepsilon  }{2a}}e_{{1}}$, $[e_2,e_3]=  {\frac {  2 {a}^{2}+\varepsilon  }{2a}}e_{{2}}$, \\
$[e_1,e_4]= ae_{{2}}$, $[e_2,e_4]= - {\frac {  2 {a}^{2}-\varepsilon  }{2a}}e_{{1}}+be_{{2}}$, \\
$[e_3,e_4]= ce_{{1}}+de_{{2}}$, \\
$a \ne 0$, $2 ad+bc \ne 0$, \\
$16 {a}^{4}+4 {a}^{2}{b}^{2}+4 {a}^{2}\varepsilon+1\ne0$. }
& \\

\cline{2-2}
& 
\parbox[c]{210pt}{\centering 
$[e_1,e_3]= {\frac {a \left( 12 {b}^{4}-11 {b}^{2}{a}^{2}+2 {a}^{4} \right) }{{b}^{4}-8 {b}^{2}{a}^{2}+{a}^{4}}}e_{{1}}+{\frac { \left( 4 {b}^{2}-{a}^{2} \right) b{a}^{2}}{{b}^{4}-8 {b}^{2}{a}^{2}+{a}^{4}}}e_{{2}}$, \\
$[e_2,e_3]= -{\frac {b \left( 5 {b}^{4}-9 {b}^{2}{a}^{2}+{a}^{4} \right) }{{b}^{4}-8 {b}^{2}{a}^{2}+{a}^{4}}}e_{{1}}+{\frac {a \left( 3 {b}^{4}-{b}^{2}{a}^{2}+{a}^{4} \right) }{{b}^{4}-8 {b}^{2}{a}^{2}+{a}^{4}}}e_{{2}}$, \\
$[e_1,e_4]= -{\frac {{b}^{3} \left( 7 {b}^{2}+2 {a}^{2} \right) }{{b}^{4}-8 {b}^{2}{a}^{2}+{a}^{4}}}e_{{1}}+{\frac {a \left( {b}^{4}-12 {b}^{2}{a}^{2}+2 {a}^{4} \right) }{{b}^{4}-8 {b}^{2}{a}^{2}+{a}^{4}}}e_{{2}}$, \\
$[e_2,e_4]= {\frac {a{b}^{2} \left( 4 {b}^{2}-{a}^{2} \right) }{{b}^{4}-8 {b}^{2}{a}^{2}+{a}^{4}}}e_{{1}}-{\frac {b \left( 8 {b}^{4}-14 {b}^{2}{a}^{2}+3 {a}^{4}\right) }{{b}^{4}-8 {b}^{2}{a}^{2}+{a}^{4}}}e_{{2}}$, \\
$[e_3,e_4]= ce_{{1}}+de_{{2}}+be_{{3}}+ae_{{4}}$, \\
${b}^{2}+{a}^{2} \ne 0$, $2 b\pm a \ne 0$, $cb-da\ne0$, \\
${b}^{4}-8 {b}^{2}{a}^{2}+{a}^{4}\ne0$, $5 {b}^{4}-13 {b}^{2}{a}^{2}+2 {a}^{4}\ne0$, \\
$30 {b}^{10}-78 {b}^{8}{a}^{2}-18 {b}^{6}{a}^{4}+78 {b}^{4}{a}^{6}-12 {b}^{2}{a}^{8}+{b}^{8}\varepsilon-16 {b}^{6}{a}^{2}\varepsilon+66 {b}^{4}{a}^{4}\varepsilon-16 {b}^{2}{a}^{6}\varepsilon+{a}^{8}\varepsilon=0$. }
& \\

\cline{2-2}
& 
\parbox[c]{210pt}{\centering 
$[e_1,e_3]= -{\frac {a \left( 3 b{c}^{2}-2 b{a}^{2}+3 {c}^{3}-2 c{a}^{2} \right) }{{c}^{3}}}e_{{1}}-{\frac {{a}^{2} \left( b+c \right) }{{c}^{2}}}e_{{2}}$, \\
$[e_2,e_3]= be_{{1}}-{\frac { \left( b{c}^{2}-2 b{a}^{2}+2 {c}^{3}-2 c{a}^{2} \right) a}{{c}^{3}}}e_{{2}}$, \\
$[e_1,e_4]= {\frac {  2 b{c}^{2}-b{a}^{2}+3 {c}^{3}-c{a}^{2} }{{c}^{2}}}e_{{1}}+{\frac {a \left( b{a}^{2}+{c}^{3}+c{a}^{2} \right) }{{c}^{3}}}e_{{2}}$, \\
$[e_2,e_4]= -{\frac {a \left( b+c \right) }{c}}e_{{1}}+{\frac {  2 b{c}^{2}-3 b{a}^{2}+2 {c}^{3}-3 c{a}^{2} }{{c}^{2}}}e_{{2}}$, \\
$[e_3,e_4]= de_{{1}}+fe_{{2}}+ce_{{3}}+ae_{{4}}$, \\
$c \ne 0$, $dc-fa \ne 0$, $b{c}^{2}-b{a}^{2}-c{a}^{2}\ne0$,\\ 
$b+c\ne0$, \\
$2  \left( c-a \right) ^{3} \left( c+a \right) ^{3}{b}^{2}+4 c \left( c-a \right) ^{3} \left( c+a \right) ^{3}b-{c}^{2} \left( 4 {c}^{4}{a}^{2}-6 {c}^{2}{a}^{4}+2 {a}^{6}-{c}^{4}\varepsilon \right) =0$. }
& \\

\cline{2-2}
& 
\parbox[c]{210pt}{\centering 
$[e_1,e_3]= \delta \sqrt {4 {a}^{2}+2 \varepsilon}e_{{1}}$, \\
$[e_2,e_3]= \left( -a+\delta \sqrt {4 {a}^{2}+2 \varepsilon} \right) e_{{2}}$, \\
$[e_1,e_4]= \left( a+\frac12 \delta \sqrt {4 {a}^{2}+2 \varepsilon} \right) e_{{2}}$, \\
$[e_3,e_4]= be_{{1}}+ce_{{2}}+ae_{{4}}$, \\
$2 {a}^{2}+\varepsilon > 0$, $c \ne 0$. }
& \\

\cline{2-2}
& 
\parbox[c]{210pt}{\centering 
$[e_1,e_3]= \frac32  {\frac { \left( 3 {a}^{2}-2 {b}^{2} \right) b}{{a}^{2}-2 {b}^{2}}}e_{{1}}+\frac32  {\frac {a{b}^{2}}{{a}^{2}-2 {b}^{2}}}e_{{2}}$, \\
$[e_2,e_3]= -\frac12  {\frac {a \left( 5 {a}^{2}-4 {b}^{2} \right) }{{a}^{2}-2 {b}^{2}}}e_{{1}}+\frac12  be_{{2}}$, \\
$[e_1,e_4]= -\frac12  {\frac {a \left( 4 {a}^{2}+{b}^{2} \right) }{{a}^{2}-2 {b}^{2}}}e_{{1}}+\frac12  {\frac { \left( 2 {a}^{2}-7 {b}^{2} \right) b}{{a}^{2}-2 {b}^{2}}}e_{{2}}$, \\
$[e_2,e_4]= \frac32  {\frac {b{a}^{2}}{{a}^{2}-2 {b}^{2}}}e_{{1}}-\frac32  {\frac { \left( 2 {a}^{2}-3 {b}^{2} \right) a}{{a}^{2}-2 {b}^{2}}}e_{{2}}$, \\
$[e_3,e_4]= ce_{{1}}+de_{{2}}+ae_{{3}}+be_{{4}}$, \\
${a}^{2}-2 {b}^{2} \ne 0$, $5 {a}^{2}-7 {b}^{2} \ne 0$, $ca-db\ne0$, \\
$ \left( {a}^{2}-{b}^{2} \right)  \left( 5 {a}^{2}-7 {b}^{2} \right)  \left( {a}^{2}+{b}^{2} \right) +2 \varepsilon \left( {a}^{2}-2 {b}^{2} \right) ^{2}=0$. }
& \\

\cline{2-2}
& 
\parbox[c]{210pt}{\centering 
$[e_1,e_3]=  {\frac {  2  \left( a+b \right) ^{2}{c}^{2}+2 {a}^{2}b \left(a+b \right) +{a}^{2}\varepsilon  }{2ac \left( a+b \right) }}e_{{1}}+ae_{{2}}$, \\
$[e_2,e_3]= be_{{1}}+{\frac { \left( 2 a+b \right) c}{a}}e_{{2}}$, \\
$[e_1,e_4]=  \left( 2 b+a \right) e_{{1}}+ce_{{2}}$, \\
$[e_2,e_4]= {\frac {cb}{a}}e_{{1}}+ {\frac {  2  \left( a+b \right)  \left( {a}^{2}+ab+{c}^{2} \right) -a\varepsilon  }{2a \left( a+b \right) }}e_{{2}}$,\\
$[e_3,e_4]= de_{{1}}+fe_{{2}}$, \\
$a \ne 0$, $c \ne 0$, $a+b \ne 0$. }
& \\

\cline{2-2}
& 
\parbox[c]{210pt}{\centering 
$[e_1,e_3]= ae_{{1}}- {\frac {  2 {b}^{2}+ \varepsilon  }{2b}}e_{{2}}$, $[e_2,e_3]= be_{{1}}$, \\
$[e_1,e_4]=  {\frac {  2 {b}^{2}- \varepsilon  }{2b}}e_{{1}}$, $[e_2,e_4]=  {\frac {  2 {b}^{2}- \varepsilon  }{2b}}e_{{2}}$, \\
$[e_3,e_4]= ce_{{1}}+de_{{2}}$, \\
$b \ne 0$, $ad+2 bc \ne 0$. }
& \\

\cline{2-2}
& 
\parbox[c]{210pt}{\centering 
$[e_1,e_3]= ae_{{1}}$, $[e_2,e_3]= be_{{1}}+ce_{{2}}$, \\
$[e_1,e_4]= 2 be_{{1}}$, $[e_2,e_4]= ce_{{1}}-{\frac {  ac-2 {b}^{2}-{c}^{2}  }{b}}e_{{2}}$, \\
$[e_3,e_4]= de_{{1}}+fe_{{2}}$, \\
$b \ne 0$, $abf-acd+2 {b}^{2}d-bcf+{c}^{2}d\ne0$, \\
$2 ac-2 {b}^{2}-2 {c}^{2}- \varepsilon=0$. }
& \\

\cline{2-2}
& 
\parbox[c]{210pt}{\centering 
$[e_1,e_3]= ae_{{1}}+be_{{2}}$, $[e_2,e_3]= ce_{{1}}+{\frac { \left( 2 b+c \right) de_{{2}}}{b}}$, \\
$[e_1,e_4]=  \left( 2 c+b \right) e_{{1}}+de_{{2}}$, $[e_2,e_4]= {\frac {dce_{{1}}}{b}}+fe_{{2}}$, \\
$[e_3,e_4]= he_{{1}}+ze_{{2}}$, \\
$b \ne 0$, ${a}^{2}+{d}^{2}+{c}^{2}+{f}^{2} \ne 0$, \\
$abd-{b}^{3}-2 {b}^{2}c+{b}^{2}f-2 b{d}^{2}-{d}^{2}c=0$, \\
$2 a{b}^{2}d+abdc-{b}^{3}c-2 {b}^{2}{d}^{2}-{b}^{2}cf-2 b{d}^{2}c-{d}^{2}{c}^{2}-{b}^{2}\varepsilon=0$. }
& \\

\hline

\end{longtable}
}

\begin{Remark}
There are $\varepsilon_i=\pm1$, $\varepsilon=\pm1$, $\delta=\pm1$ and  $\delta_i=\pm1$ in the~Table~\ref{table:2}.
\end{Remark}


\begin{proof}
Let $\rho$ has Segre type $\{111\overline{1}\}$. Then there exists a basis $\{e_1,e_2,e_3,e_4\}$ in the metric Lie algebra such that the metric tensor $g$ and the Ricci tensor~$r$ have the following form (see~\cite{ONeil})
\begin{equation} \label{eq:rg_in_1111bar}
r = \begin {pmatrix}
\varepsilon_1\rho_1 & 0                   & 0                   & 0 \\
0                   & \varepsilon_2\rho_1 & 0                   & 0 \\
0                   & 0                   & \varepsilon_3\alpha & \varepsilon_3\beta   \\
0                   & 0                   & \varepsilon_3\beta  & -\varepsilon_3\alpha
\end {pmatrix},\quad 
g = \begin {pmatrix}
\varepsilon_1 & 0             & 0             & 0 \\
0             & \varepsilon_2 & 0             & 0 \\
0             & 0             & \varepsilon_3 & 0 \\
0             & 0             & 0             & \varepsilon_3
\end {pmatrix},
\end{equation}
where $\rho_1\ne\rho_2$, $\beta\ne0$ and $\varepsilon_i=\pm1$.

Next, we consider the equations~(\ref{eq:r_ijk=r_ikj}) as restrictions on the structure constants of the Lie algebra:
\begin{gather*}
\begin{aligned} \left( \rho_1-\alpha \right) C_{14}^1-\beta C_{13}^1 &=0, & \left( \rho_1-\rho_2 \right) C_{12}^1 &=0,\\
\left( \rho_1-\alpha \right) C_{13}^1+\beta C_{14}^1 &=0, & \left( \rho_1-\rho_2 \right) C_{12}^2 &=0,\\
\left( \rho_2-\alpha \right) C_{24}^2-\beta C_{23}^2 &=0, & \beta C_{34}^3 &=0,\\
\left( \rho_2-\alpha \right) C_{23}^2+\beta C_{24}^2 &=0, & \beta C_{34}^4 &=0, \end{aligned}\\
\left( \beta \left( C_{13}^4+3 C_{14}^3 \right) +2 \left( \rho_1-\alpha \right) C_{14}^4 \right) \varepsilon_{{3}}+ \beta C_{34}^1\varepsilon_{{1}}=0,\\
\left( \beta \left( 3 C_{13}^4+C_{14}^3 \right) -2 \left( \rho_1-\alpha \right) C_{13}^3 \right) \varepsilon_{{3}}+ \beta C_{34}^1\varepsilon_{{1}}=0,\\
\left( \beta \left( C_{23}^4+3 C_{24}^3 \right) +2 \left( \rho_2-\alpha \right) C_{24}^4 \right) \varepsilon_{{3}}+ \beta C_{34}^2\varepsilon_{{2}}=0,\\
\left( \beta \left( 3 C_{23}^4+C_{24}^3 \right) -2 \left( \rho_2-\alpha \right) C_{23}^3 \right) \varepsilon_{{3}}+ \beta C_{34}^2\varepsilon_{{2}}=0,\\
\left( \left( \rho_1+\rho_2-2 \alpha \right) C_{12}^3-2 \beta C_{12}^4 \right) \varepsilon_{{3}}+ \left( \rho_1-\rho_2 \right)  \left(C_{13}^2\varepsilon_{{2}}+C_{23}^1\varepsilon_{{1}} \right) =0,\\
\left(  \left( \rho_1+\rho_2-2 \alpha \right) C_{12}^4+2 \beta C_{12}^3 \right) \varepsilon_{{3}}- \left( \rho_1-\rho_2 \right)  \left(C_{14}^2\varepsilon_{{2}}+C_{24}^1\varepsilon_{{1}} \right) =0,\\
\left( \left( \alpha - \rho_1 \right) \left( C_{13}^4-C_{14}^3 \right) -2 \beta C_{13}^3 \right)\varepsilon_{{3}}- \left( \alpha-\rho_1 \right) C_{34}^1 \varepsilon_{{1}}=0,\\
\left( \left( \alpha - \rho_1 \right) \left( C_{13}^4-C_{14}^3 \right) -2 \beta C_{14}^4 \right)\varepsilon_{{3}}+ \left( \alpha-\rho_1 \right) C_{34}^1 \varepsilon_{{1}}=0,\\
\left( \left( \alpha - \rho_2 \right) \left( C_{24}^3-C_{23}^4 \right) +2 \beta C_{23}^3 \right) \varepsilon_{{3}}+ \left( \alpha-\rho_2 \right) C_{34}^2 \varepsilon_{{2}}=0,\\
\left( \left( \alpha - \rho_2 \right) \left( C_{24}^3-C_{23}^4 \right) +2 \beta C_{24}^4 \right) \varepsilon_{{3}}- \left( \alpha-\rho_2 \right) C_{34}^2 \varepsilon_{{2}}=0,\\
\left( \rho_1-\alpha \right) C_{34}^1 \varepsilon_{{1}}-\beta \left( C_{13}^3-C_{14}^4 \right) \varepsilon_{{3}} =0,\\
\left( \rho_2-\alpha \right) C_{34}^2 \varepsilon_{{2}}-\beta \left( C_{23}^3-C_{24}^4 \right) \varepsilon_{{3}} =0,\\
\begin{split}\left( \left( \rho_2-2 \rho_1+\alpha \right) C_{23}^1-\beta C_{24}^1 \right) \varepsilon_{{1}}&+ \left( \alpha C_{13}^2-\rho_2 C_{13}^2-\beta C_{14}^2 \right) \varepsilon_{{2}}+\\
&+ \left( \alpha C_{12}^3-\rho_2 C_{12}^3+\beta C_{12}^4 \right) \varepsilon_{{3}} =0,\end{split}\\
\begin{split}\left( \left( \rho_2-2 \rho_1+\alpha \right) C_{24}^1+\beta C_{23}^1 \right) \varepsilon_{{1}}&+ \left( \alpha C_{14}^2-\rho_2 C_{14}^2+\beta C_{13}^2 \right) \varepsilon_{{2}}-\\
&- \left( \alpha C_{12}^4-\rho_2 C_{12}^4-\beta C_{12}^3 \right) \varepsilon_{{3}} =0,\end{split}\\
\begin{split}\left( \left( \rho_1-2 \rho_2+\alpha \right) C_{13}^2-\beta C_{14}^2 \right) \varepsilon_{{2}}&+ \left( \alpha C_{23}^1-\rho_1 C_{23}^1-\beta C_{24}^1 \right) \varepsilon_{{1}}-\\
&- \left( \alpha C_{12}^3-\rho_1 C_{12}^3+\beta C_{12}^4 \right) \varepsilon_{{3}} =0,\end{split}\\
\begin{split}\left( \left( \rho_1-2 \rho_2+\alpha \right) C_{14}^2+\beta C_{13}^2 \right) \varepsilon_{{2}}&+ \left( \alpha C_{24}^1-\rho_1 C_{24}^1+\beta C_{23}^1 \right) \varepsilon_{{1}}+\\
&+ \left( \alpha C_{12}^4-\rho_1 C_{12}^4-\beta C_{12}^3 \right) \varepsilon_{{3}} =0.\end{split}
\end{gather*}

Further, we use the Jacobi identity, and we calculate the matrix of the Ricci tensor and equate the resulting matrix components to the components of the matrix~(\ref{eq:rg_in_1111bar}). Solving the resulting system with the help of Gr\"{o}bner bases~\cite{12}, within the constraints $\rho_1\ne\rho_2$, $\beta\ne0$, and discarding conformally flat and Ricci parallel cases, we find that
\begin{gather*}
C_{1, 2}^1 = C_{1, 2}^2 = C_{1, 2}^3 = C_{1, 2}^4 = 0,\\ 
C_{1, 3}^1 = C_{1, 3}^2 = C_{1, 3}^3 = C_{1, 3}^4 = 0,\\
C_{1, 4}^1 = C_{1, 4}^2 = C_{1, 4}^3 = C_{1, 4}^4 = 0,\\
C_{2, 3}^1 = C_{2, 3}^2 = 0,\quad C_{2, 4}^1 = C_{2, 4}^2 = 0,\\
C_{3, 4}^1 = C_{3, 4}^3 = C_{3, 4}^4 = 0,\\
C_{2, 3}^3 = -C_{2, 4}^4 = -a\delta\sqrt{3}, \quad C_{2, 3}^4 = C_{2, 4}^3 = a, \quad C_{3, 4}^2 = 2a\varepsilon_2\varepsilon_3, \\
\rho_1 = 0, \quad \rho_2 = -8a^2\varepsilon_2, \quad \alpha = 4a^2\varepsilon_2, \quad \beta = 4a^2\delta\varepsilon_2\sqrt{3},
\end{gather*}
where $\delta=\pm1$, $a\ne0$.

Cases of others Segre types are considered by analogy.
\end{proof}

\begin{Remark}
Our next step is to determine which metric Lie algebra in the Table~\ref{table:2} are isomorphic to each other.
\end{Remark}

\end{document}